\documentclass[12pt, a4paper]{article}
\usepackage{float}
\usepackage{amsmath}
\usepackage{amsthm,amscd,amsfonts,yhmath,mathrsfs,calrsfs}
\usepackage{amssymb, upref, color}
\usepackage{graphicx}
\usepackage[section]{placeins}

\usepackage[margin=2.3cm]{geometry}

\numberwithin{equation}{section}

\newtheorem*{theoremA}{Theorem A}

\newtheorem{lemm}{Lemma}[section]

\makeatletter
\renewenvironment{proof}[1][\indent\proofname]{\par
  \pushQED{\qed}%
  \normalfont \topsep6\p@\@plus6\p@\relax
  \trivlist
  \item[\hskip\labelsep
        \bfseries
    #1\@addpunct{.}]\ignorespaces
}{\popQED\endtrivlist\@endpefalse}
\makeatother
\begin{document}

\title{Finite groups with some bounded codegrees
}
\author{
Dongfang Yang$^a$, Yu Zeng$^b$, Heng Lv$^a$\footnote{Corresponding author.  Email: lvh529@163.com
\newline \indent ~~2010 AMS Mathematics Subject Classification: 20C15 
\newline \indent ~~This research was supported by the National Natural Science Foundation of China (No.11971391, 12071376), the first and second authors are supported by the Chinese Scholarship Council.
}
\\
$^a$\small{School of Mathematics and Statistics, Southwest University, Chongqing, China}\\
$^b$\small{Department of Mathematics, Changshu Institute of Technology, Jiangsu, China}}
\date{}
\maketitle

    \noindent\textbf{Abstract.} For a character $\chi$ of a finite group $G$, the number cod$(\chi):=|G:\mathrm{ker}(\chi)|/\chi(1)$ is called the codegree of $\chi$.
     In this paper, we give a solvability criterion for a finite group $G$ depending on the minimum of the ratio $\chi(1)^2 /\mathrm{cod}(\chi )$, when $\chi $ varies among the irreducible characters of $G$.



\section{Introduction}
Throughout all groups are finite and $p$ is a prime number.  Let $G$ be a group. We denote by Irr$(G)$ the set of irreducible characters of $G$; and use cd$(G)$ to denote the set of all irreducible character degrees of $G$.
  Let $\chi$ be a character of $G$. Qian, et al. \cite{Qian 2007} called the positive integer cod$(\chi):=|G:\rm{ker}(\chi)|/\chi(1)$ the codegree of the character $\chi$ and studied the connection between the structure of a group and the codegrees of its irreducible characters (see the details in \cite{Qian 2016}, \cite{Qian 2011}, \cite{Qian 2012}).
In these papers, they studied the groups through considering the prime divisors of the codegrees of its irreducible characters.
In \cite{Yang}, Yang et al. obtained that the group $G$ is solvable if cod$(\chi)\le p_{\chi}\cdot \chi(1)$ for each nonlinear irreducible character $\chi$ of a group $G$, where $p_{\chi}$ is the largest prime divisor of $|G:\mathrm{ker}(\chi)|$.
Notice that the codegree is larger than its degree for any nonprincipal irreducible character of $G$ (for any nonprincipal irreducible character $\chi $ of $G$, $\chi \in \mathrm{Irr}(G/\ker(\chi ))$, so that $|G/\ker(\chi )|> \chi (1)^2$ since $G/\ker(\chi)>1$).
So, it is natural to ask that what positive real number $k$ such that
\[
k\cdot \mathrm{cod}(\chi)\le \chi(1)^{2}, \tag{$*$}
\]
for all nonlinear irreducible characters $\chi$ of $G$ guarantees the solvability of $G$.
In the next theorem, we give a sharp lower bound of $k$.

\begin{theoremA}
    Let $G$ be a finite group with $k\cdot$$\rm{cod}$$(\chi)$$\le\chi(1)^2$ for every nonlinear irreducible character $\chi$ of $G$. If $k>a=\frac{2^9\cdot3^2\cdot19^2}{5\cdot 7^3\cdot11\cdot31}$, then $G$ is solvable. The real number $a$ is best possible.
\end{theoremA}

The positive real number $a$ in the above theorem is best possible.
For instance, let $G\cong O'N$, and let $\theta\in\mathrm{Irr}(G)$ be of minimal degree $10944(=2^6\cdot3^2\cdot19)$, we get
\[
a\cdot|G|=2^{18}\cdot3^6\cdot19^3=\theta(1)^3.
\]
It forces $a\cdot|G|\le \theta(1)^3$ for every nonlinear irreducible character $\theta$ of a nonabelian simple group $G$.


More generally, given a positive integer $s$, one could ask what positive real number $k$ such that
\[
k\cdot \mathrm{cod}(\chi)\le \chi(1)^{s}
\]
for all nonlinear irreducible characters $\chi$ of $G$ guarantees the solvability of $G$.


\section{Proof}

At first, we gather some lemmas before we actually prove the main theorem.

\begin{lemm}\label{Kay}
   Suppose that $N$ is a minimal normal nonabelian subgroup of a group $G$. Then there exists an irreducible character  $\theta$ of $N$ such that $\theta$ is extendible to $G$ with $\theta(1)\ge 5$.
\end{lemm}
\begin{proof}
See \cite[Theorem 1.1]{Kay 2011}.
\end{proof}

We also need the following result regarding the characters of simple groups of Lie type.

\begin{lemm}\label{Tong}
Let $S$ be a simple group of Lie type in characteristic $p$ defined over a field of size $q$. Assume that $S\ncong  A_1(q)$, $^2F_4(2)'$. Then there exist two irreducible characters $\theta_i$, $i=1,2$, of $S$ such that both $\theta_i$ extend to $\mathrm{Aut}(S)$ with $1<\theta_1(1)<\theta_2(1)$ and $\theta_2(1)=|S|_p$.
\end{lemm}
\begin{proof}
See \cite[Lemma 2.4]{Tong}.
\end{proof}

Our proof will be based on the classification of finite simple groups. The main point of the classification is that if $S$ is a nonabelian simple group, then $S$ is either a simple group of Lie type, an alternating group of degree at least 5, one of 26 sporadic simple groups, or the Tits group. Often the Tits group, which denoted by $^2F_4(2)'$, is lumped in the groups of Lie type. Since the Tits group is in the Atlas \cite{Atlas}, here we will consider this simple group separately.

For the simple groups of Lie type, there is a particular character called Steinberg character. 
If $S$ is a simple group of Lie type with defining characteristic $p$ and $\theta$ is the Steinberg character of $S$, then $\theta(1)=|S|_p$.
By Lemma \ref{Kay}, if $G$ has a normal subgroup $S$ such that $S$ is isomorphic to a finite simple group of Lie type, then the Steinberg character of $S$ extends to $G$.  
Actually, the character $\theta_2$ in Lemma \ref{Tong} is the Steinberg character (see for instance \cite[Proposition 3.1]{Malle 2020}).

For the alternating groups $A_n$, where $n\geq5$, the automorphism group of $A_n$ is $S_n$ except $n=6$. In \cite[Theorem 3]{Bianchi 2007}, it is proved that $A_n$, where $n\ge6$, has a character of degree $\frac{n(n-3)}{2}$ which is extendible to Aut$(A_n)$.

 For the sporadic simple groups, irreducible character degrees are available in the Atlas $\cite{Atlas}$. In Table \ref{Sporadic}, we list some character degrees of sporadic simple groups
 which will be relevant to our problems.

\begin{lemm}\label{P1}
Let $S$ be a simple group and $k>a=\frac{2^9\cdot3^2\cdot19^2}{5\cdot 7^3\cdot11\cdot31}$. If $S$ is not isomorphic to one of the groups in $\{A_1(q), G_2(q), ^2G_2(q^2), Fi_{22}\}$, then $S$ has some nontrivial irreducible character $\theta$ that is extendible to $\mathrm{Aut}(S)$ and such that $k\cdot|S|>\theta(1)^3$.
\end{lemm}
For a simple group of Lie type $S$, we know from \cite[Lemma 2.6]{Ah} that $|S|<\theta(1)^3$ for the Steinberg character $\theta$ (which is actually the $\theta_2$ in Lemma \ref{Tong}). By Lemma \ref{Tong}, it is natural to check whether the character $\theta_1$ of $S$ satisfies $k\cdot|S|>\theta_1(1)^3$ or not.
Based on the proof of \cite[Lemma 2.4]{Tong}, Table \ref{Lie} lists $\theta_1(1)$ for every simple group of Lie type in Lemma \ref{P1}.\\

\textbf{Proof~of~Lemma~\ref{P1}} Observe that $k>a=\frac{2^9\cdot3^2\cdot19^2}{5\cdot 7^3\cdot11\cdot31}>\frac{5}{2}$.

  Assume that $S$ is an alternating group $A_n$, where $n\ge5$. By \cite[Theorem 3]{Bianchi 2007}, $S$ has some irreducible character $\theta$ of degree $\frac{n(n-3)}{2}$ that extends to Aut$(A_n)$. Furthermore, we have
\[
\frac{k\cdot|A_n|}{\theta(1)^3}=\frac{4k\cdot(n-1)(n-2)\cdot(n-4)!}{n^2(n-3)^2}.
\]
 Note that $(n-1)(n-2)>n(n-3)$. If $n\ge6$, then $4k\cdot(n-4)!>n(n-3)$ and hence $k\cdot|A_n|>\theta(1)^3$.
If $n=5$, $A_5$ has an irreducible character $\theta$ of degree 4 that is extendible to Aut$(A_5)$, and we have $k\cdot|A_5|>2^3\cdot 3\cdot 5>4^3=\theta(1)^3$, and we are done for the alternating groups $A_n$ of degree at least 5.

Assume that $S$ is a sporadic simple group other than $Fi_{22}$.
By \cite[Theorem 4]{Bianchi 2007}, there exists some irreducible character $\theta\in\mathrm{Irr}(S)$ of degree as in Table \ref{Sporadic} that is extendible to $\mathrm{Aut}(S)$. Also, it's easy to check $k\cdot|S|>\theta(1)^3$, we are also done for the sporadic simple groups but other than $Fi_{22}$.

Assume that $S\cong $ $^2F_4(2)'$, by the Atlas \cite{Atlas}, $S$ has an irreducible character $\theta$ of degree $27$ that is extendible to $\mathrm{Aut}(S)$. Obviously, we also have $k\cdot|S|=k\cdot 2^{11}\cdot3^3\cdot 5^2\cdot 13>27^3=\theta(1)^3$.

Assume that $S$ is a simple group of Lie type and $S\notin\{A_1(q), G_2(q), ^2G_2(q^2)\}$.
Take $\theta\in\mathrm{Irr}(S)$ as the irreducible character $\theta_1$ in Lemma \ref{Tong}, we have that $\theta$ is extendible to $\mathrm{Aut}(S)$. Next we will check the stronger conclusion that $|S|>\theta(1)^3$ and hence the result follows immediately.

If $S\cong A_n(q)$, where $n\ge2$, then
\[
\begin{split}
|A_n(q)|&=\frac{1}{(n+1,q-1)}q^{\frac{n(n+1)}{2}}(q^2-1)(q^3-1)\cdots(q^{n+1}-1)
\\
&> \frac{1}{q-1}q^{\frac{n(n+1)}{2}}q\cdot q^2\cdots q^n=\frac{q^{n(n+1)}}{q-1}.
\end{split}
\]
By Table \ref{Lie}, choose $\theta\in\mathrm{Irr}(S)$ with $\theta(1)=\frac{q(q^n-1)}{q-1}<\frac{q^{n+1}}{q-1}$, and we have
\[
\frac{|S|}{\theta(1)^3}> \frac{q^{n(n+1)}(q-1)^3}{q^{3n+3}(q-1)}=q^{n^2-2n-3}\cdot(q-1)^2
=q^{(n-1)^2-4}(q-1)^2.
\]
If $n\ge3$, it is obvious that $(n-1)^2-4\ge0$. That means $|S|>\theta(1)^3$.
If $n=q=2$, then there exists a $\theta \in \mathrm{Irr}(S)$ of degree $3$ such that $|S|>\theta(1)^3$ by \cite{Atlas}.
If $n=2$ and $q\ge3$, then since $|S|=|A_2(q)|\ge\frac{1}{q-1}q^3(q^3-1)(q^2-1)=q^3(q^3-1)(q+1)$ and $\theta(1)=q(q+1)$ we have
\[
\frac{|S|}{\theta(1)^3}> \frac{q^3(q^3-1)(q+1)}{q^3(q+1)^3}=\frac{q^3-1}{(q+1)^2}>1.
\]
This yields $|S|>\theta(1)^3$.

For every remaining simple group of Lie type, it is routine to check $|S|>\theta(1)^3$ for character $\theta$ of degree as in Table \ref{Lie} (The similar calculation process is omitted here). Therefore, the proof is complete.
~~~~~~~~~~~~~~~~~~~~~~~~~~~~~~~~~~~~~~~~~~~~~~~~~~~~~~~~~~~~~~~~~~~~~~~~~~~~~~~~~~~~~~~~~~~~~~~~$\Box$
\\

Let $H$ be a subgroup of a finite group $G$, and let $\theta $ be an irreducible character of $H$. We use $\mathrm{Irr}(G|\theta )$ to denote the set of irreducible characters of $G$ lying over $\theta $.

\begin{lemm}\label{Simple}
	Let $G$ be a group and $N$ be its nonsolvable minimal normal subgroup. Suppose that $\rm{cod}$$(\chi)$$\le \chi(1)^2$ for each nonlinear irreducible character $\chi$ of $G$. Then $N$ is nonabelian simple.
\end{lemm}
\begin{proof} 
    Let $N=S_1\times \cdots\times S_t$ where $S_i\cong S$ for a nonabelian simple group $S$.
    By the unique minimality of $N$, it follows that $G$ acts transitively on $\{S_1,S_2,\cdots, S_t\}$. Let $S=S_1$. Then $S\unlhd N_G(S)$ and $t=|G:N_G(S)|$.
    By Lemma \ref{Kay}, $S$ has some irreducible nonlinear character $\theta$ such that $\theta$ is extendible to $N_G(S)$. Take $\chi\in \mathrm{Irr}(G|\theta)$. Then $\chi(1)\le |G:N_G(S)|\cdot\theta(1)=t\cdot \theta(1)$.

    Note that ker$(\chi)=1$. So, we conclude from the hypothesis that $|G|\le \chi(1)^3$. Furthermore,

\[
t^3\cdot\theta(1)^3\ge\chi(1)^3\ge |G|\ge |G:N_G(S)|\cdot|N|=t\cdot|S|^t.
\]
Since $\theta(1)^2<|S|$, we have
\[
t^2\ge\frac{|S|^t}{\theta(1)^3}>|S|^{t-\frac{3}{2}}.
\]
If $t\ge 3$, then since $t-\frac{3}{2}\ge \frac{t}{2}$ we have $t^2>|S|^{t-\frac{3}{2}}\ge |S|^{\frac{t}{2}}=\sqrt{|S|}^t>7^t$, a contradiction. Similarly, if $t=2$, we have $4>\sqrt{|S|}>7$, which is also a contradiction. Thus $t=1$.
\end{proof}

\begin{lemm}\label{arith}
 If $q=p^f$ is a power of a prime $p$ such that $q \equiv 3 (\mathrm{mod}~6)$, then $f< q^{\frac{1}{3}}$.
\end{lemm}
\begin{proof}
    If $f=1$, we are done immediately. If $f=2$, it follows from $q\ge 9$ that $2\le q^{\frac{1}{3}}$. Let function $T(x)=p^{\frac{x}{3}}-x$, where $x\ge3$. Then the derivative function $T'(x)=\frac{\ln~p}{3}\cdot p^{\frac{x}{3}}-1\ge \ln~p-1\ge0$. That means $T(f)\ge T(0)>0$. That is, $f< p^{\frac{f}{3}}=q^{\frac{1}{3}}$.
\end{proof}

Recall that $\Phi_n(x)$ is the $n$-th cyclotomic polynomial with indeterminate $x$.
For a simple group of Lie type over finite field of $q$ elements, we define $\Phi_n=\Phi_n(q)$.
Next we will prove Theorem A.\\

\textbf{Proof~of~the~Theorem~A.}
	Suppose that $G$  is a nonsolvable group of minimal order which satisfies the hypothesis of Theorem A.

    We claim that $G$ is not simple.
    Suppose not. If $G\notin\{ A_1(q), G_2(q), ^2G_2(q^2),Fi_{22}\}$, by Lemma \ref{P1}, there exists some nonlinear irreducible character $\chi$ of $G$ such that $k\cdot|G|>\chi(1)^3$. That is, $k\cdot$cod$(\chi)>\chi(1)^2$, which is contrary to the hypothesis.
     If $G$ is isomorphic to one of the group in $\{ A_1(q), G_2(q), ^2G_2(q^2),Fi_{22}\}$, similarly, to get a contradiction, we just need to show that there also exists some nonlinear irreducible character $\theta$ of $G$ such that $k\cdot|G|>\theta(1)^3$.

    Suppose that $G\cong A_1(q)$.
    Let $\chi\in \mathrm{Irr}(G)$ be of degree $q-1$. Then since $k>\frac{5}{2}$ we have
    \[
    k\cdot|G|>q(q^2-1)>(q-1)^3=\chi(1)^3.
    \]

    Suppose that $G\cong G_2(q)$. Since $G_2(2)\cong U_3(3).2$ is not simple, we can assume that $q\ge3$. Note that the character degrees of $G_2(q)$ are listed in the paper \cite{Chang}.
     If $q\equiv1(\mathrm{mod}~6)$ or $q\equiv 4(\mathrm{mod}~6)$, then $\Phi_2\Phi_6\in\mathrm{cd}(G)$.
       If $q\equiv2(\mathrm{mod~6})$ or $q\equiv5(\mathrm{mod}~6)$, then $\Phi_1\Phi_3\in\mathrm{cd}(G)$.
       If $q\equiv3(\mathrm{mod~6})$, then $\Phi_3\Phi_6\in\mathrm{cd}(G)$.
      It follows that
    \[
    \frac{k\cdot|G|}{(\Phi_3\Phi_6)^3}=\frac{3q^6(q^6-1)(q^2-1)}{(q^2+q+1)^3(q^2-q+1)^3}>\frac{3(q^2-1)^4}{q^6-1}>1.
    \]
 Obviously, $\Phi_3\Phi_6>\Phi_2\Phi_6>\Phi_1\Phi_3$. So there always exists some character $\chi$ of $G$ such that $k\cdot|G|>\chi(1)^3$.

    Suppose that $G\cong$$^2G_2(q^2)$, where $q^2=3^{2m+1}$. Taking $\chi\in \mathrm{Irr}(G)$ with degree $\Phi_{12}$. Then
    \[
\frac{k\cdot|G|}{{\Phi_{12}}^3}\ge\frac{q^6(q^6+1)(q^2-1)\cdot (q^2+1)^3}{(q^4+q^2+1)^3\cdot (q^2+1)^3}\ge \frac{(q^2-1)(q^2+1)^3}{q^6-1}>1.
    \]

    Suppose that $G\cong Fi_{22}$. By the Atlas \cite{Atlas}, it is easy to get that $k\cdot|G|>\chi(1)^3$ for some nonlinear $\chi\in\mathrm{Irr}(G)$.

    We next conclude the final contradiction.
     Note that $G$ is not simple.
     Assume that $G$ has distinct minimal normal subgroups, says $N_1, N_2$. It follows from the minimality of $G$ that $G/N_i$ is solvable for $i=1,2$, and then $G\lesssim G/N_1\times G/N_2$ is solvable, a contradiction.
     So we may assume that $G$ has a unique minimal normal subgroup $N$; and, similarly, $G/N$ is solvable; and this means that $N$ is also nonsolvable.
     Hence $N$ is a direct product of isomorphic nonabelian simple groups.
     Since $\mathrm{cod}(\chi)<k\cdot$cod$(\chi)\le\chi(1)^2$ for every nonlinear irreducible character $\chi$ of $G$, it follows from Lemma \ref{Simple} that the minimal normal subgroup $N$ of $G$ is simple.
     If $N\notin\{A_1(q), G_2(q), $$^2$$G_2(q^2),Fi_{22}\}$, then, by Lemma \ref{P1}, $N$ has some nontrivial irreducible character $\theta$ that is extendible to Aut$(N)$(also to $G$) and $k\cdot|N|>\theta(1)^3$. Let $\chi\in\mathrm{Irr}(G)$ be such that $\chi_N=\theta$. Then
     \[
     k\cdot|G|>k\cdot|N|>\theta(1)^3=\chi(1)^3.
     \]
     Observe that ker$(\chi)=1$. This means $k\cdot \mathrm{cod}(\chi)>\chi(1)^2$, a contradiction. Next we will also get a contradiction when $N$ is isomorphic to one of the groups in $\{ A_1(q), G_2(q), ^2G_2(q^2),Fi_{22}\}$.

    (1) Suppose that $N\cong A_1(q)$. Let $\theta\in \mathrm{Irr}(N)$ be the Steinberg character. Then there exists $\chi\in\mathrm{Irr}(G)$ such that $\chi_N=\theta$. Thus
     \[
     k\cdot|G|\ge 2k\cdot|N|\ge kq(q^2-1)>q^3=\theta(1)^3=\chi(1)^3.
     \]

    (2) Suppose that $N\cong G_2(q)$, where $q=p^f$. Then $|N|=q^6(q^6-1)(q^2-1)$.

  ~If $q\not\equiv 3(\rm{mod}~6)$,
     then, by the Atlas \cite{Atlas}, we obtain
     $|G:N|\le |\mathrm{Out}(N)|\le q$.
     Also, there exists some $\theta\in\mathrm{Irr}(N)$ of degree at most $q^3+1$ (see for instance \cite{Chang}). Let $\chi\in\mathrm{Irr}(G|\theta)$. It follows that $\chi(1)\le |G:N|\cdot \theta(1)\le  q(q^3+1)$. Hence
     \[
     \frac{k\cdot|G|}{\chi(1)^3}\ge \frac{k\cdot q^6(q^6-1)(q^2-1)}{q^3(q^3+1)^3}=\frac{k\cdot q^3(q^3-1)(q^2-1)}{(q^3+1)^2}.
     \]
Note that $k>a>\frac{5}{2}$. Since $k\cdot q^3(q-1)>q^3+1$ and $(q^3-1)(q+1)>q^3+1$, then we have $k\cdot |G|>\chi(1)^3$. 

If $q\equiv 3(\mathrm{mod}~6)$,
 then $|G:N|\le |\mathrm{Out}(N)|\le 2\cdot f< 2\cdot q^{\frac{1}{3}}$ where the last inequality follows from Lemma \ref{arith}. Now let $\theta\in\mathrm{Irr}(N)$ be of degree $(q^2+q+1)(q^2-q+1)$ and $\chi\in\mathrm{Irr}(G|\theta)$. Then $\chi(1)< 2\cdot q^{\frac{1}{3}}(q^2+q+1)(q^2-q+1)$. Thus
     \[
     \frac{k\cdot|G|}{\chi(1)^3}\ge\frac{k\cdot q^6(q^6-1)(q^2-1)}{8q\cdot(q^2+q+1)^3(q^2-q+1)^3}
     =\frac{k\cdot q^5(q^2-1)^4}{8(q^6-1)^2}=\frac{k\cdot q^5(q+1)^4(q-1)^4}{8(q^6-1)^2}.
     \]
 Since $q^5(q+1)>q^6-1$ and $(q+1)^3>q^3+1$, it follows that
     \[
     \frac{k\cdot|G|}{\chi(1)^3}\ge\frac{k(q+1)^3(q-1)^4}{8(q^6-1)}\ge\frac{k(q-1)^4}{8(q^3-1)}.
     \]
 Recall that $q\equiv 3(\mathrm{mod}~6)$. If $q\ge9$, it is easy to get $k\cdot|G|>\chi(1)^3$. If $q=3$,
     \[
     \frac{k\cdot|G|}{\chi(1)^3}\ge\frac{2k\cdot|N|}{\chi(1)^3}=
     \frac{k\cdot2^4\cdot3^6(3^6-1)}{2^3\cdot3\cdot13^3\cdot7^3}
     =\frac{2k\cdot3^5(3^6-1)}{13^3\cdot7^3}>1.
     \]

   (3) Suppose that $N\cong$$^2G_2(q^2)$, where $q^2=3^{2m+1}$. Since $q^2=3^{2m+1}\equiv 3(\mathrm{mod}~6)$, it follows from Lemma \ref{arith} that $|G:N|\le|\mathrm{Out}(N)|=2m+1< q^{\frac{2}{3}}$. Let $\theta\in\mathrm{Irr}(N)$ be of degree $\Phi_{12}$ and $\chi\in\mathrm{Irr}(G|\theta)$. Then $\chi(1)\le |G:N|\cdot\theta(1)\le q^{\frac{2}{3}}\cdot\Phi_{12}$. Thus
    \[
\frac{3\cdot|G|}{\chi(1)^3}\ge\frac{q^6(q^6+1)(q^2-1)\cdot (q^2+1)^3}{q^2\cdot(q^4+q^2+1)^3\cdot (q^2+1)^3}\ge \frac{(q^2-1)(q^2+1)^3}{q^2(q^6-1)}>1.
    \]

   (4) Suppose that $N\cong Fi_{22}$. By the Atlas \cite{Atlas}, $N$ has an irreducible character $\theta$ of degree $78$. Let $\chi\in\mathrm{Irr}(G|\theta)$. Then since Out$(N)=2$ we get $\chi(1)\le |G:N|\cdot\theta(1)=156=2^2\cdot3\cdot13$. It follows that
     \[
     k\cdot|G|\ge2k\cdot|N|=k\cdot2^{18}\cdot3^9\cdot5^2\cdot7\cdot11\cdot13>(2^2\cdot3\cdot13)^3=\chi(1)^3.
     \]


     Therefore we can find an irreducible character $\chi$ of $G$ such that $k\cdot|G|>\chi(1)^3$ when $N\in\{ A_1(q), G_2(q), ^2G_2(q^2),Fi_{22}\}$. Since $[\chi_N,1_N]=0$, it follows from the minimality of $N$ that ker$(\chi)=1$. So $k\cdot \mathrm{cod}(\chi)>\chi(1)^2$, a contradiction. The proof is established. ~~~~~~~~~~~~$\Box$



\begin{table}[htpb]
	\centering
	\caption{some degrees of some simple groups of Lie type}\label{Lie}
	\begin{tabular}{ccc|ccc}
		\hline
		Group       &~~ sym.  &~~ Degree &	Group       &~~ sym.  &~~ Degree  \\
		\hline
		$A_n(q)(n\ge2)$	&~~ $(1,n)$ &~~  $\frac{q^{n+1}-q}{q-1}$&
\par~
        $F_4(q)$	    &~~~ $\theta_{9,2}$ &~~  $q^2\Phi_3^2\Phi_6^2\Phi_{12}$\\

		$^2A_n(q)$	    &~~ $(1,n)$ &~~  $\frac{q^{n+1}+(-1)^{n+1}q}{q+1}$&
        $E_6(q)$	    &~~ $\theta_{6,1}$ &~~  $q\Phi_8\Phi_9$\\

		$B_n(q)$	  &~~ $\begin{pmatrix} 0 & 1 & n \\  & - &  \end{pmatrix}$ &~~  $\frac{(q^n-1)(q^n-q)}{2(q+1)}$&
        $^2E_6(q^2)$	    &~~ $\theta_{2,4'}$ &~~  $q\Phi_8\Phi_{18}$\\

		$C_n(q)$	&~~ $\begin{pmatrix} 0 & 1 & n \\  & - &  \end{pmatrix}$ &~~  $\frac{(q^n-1)(q^n-q)}{2(q+1)}$
        &$E_7(q)$	    &~~ $\theta_{7,1}$ &~~  $q\Phi_7\Phi_{12}\Phi_{14}$\\

		$D_n(q)$  	&~~ $\begin{pmatrix} n-1 \\  1  \end{pmatrix}$ &~~  $\frac{(q^n-1)(q^{n-1}+q)}{q^2-1}$
		&$E_8(q)$	&~~ $\theta_{8,1}$ &~~  $q\Phi_4^2\Phi_8\Phi_{12}\Phi_{20}\Phi_{24}$\\

		$^2D_n(q^2)$	&~~ $\begin{pmatrix} 1&  n-1 \\  &-  \end{pmatrix}$ &~~  $\frac{(q^n+1)(q^{n-1}-q)}{q^2-1}$
		&$^2B_2(q^2)$  	&~~ $^2B_2[a]$ &~~  $\frac{1}{\sqrt{2}}q\Phi_1\Phi_2$\\

		$^3D_4(q^3)$     &~~ $\theta_{1,3'}$ &~~  $q\Phi_{12}$
		&$^2F_4(q^2)$	&~~ $^2B_2[a]$ &~~  $\frac{1}{\sqrt{2}}q\Phi_1\Phi_2\Phi_4^2\Phi_6$\\

		\hline
	\end{tabular}

\end{table}

\begin{table}[htpb]
	\centering
	\caption{degrees of the sporadic simple groups}\label{Sporadic}
	\begin{tabular}{ccc|ccc}
		\hline
		Group       &~~~~ char.  &~~~~ Degree &	Group       &~~~~ char.  &~~~~ Degree  \\
		\hline
		$M_{11}$	&~~~~ $\chi_2$ &~~~~  $10=2\cdot 5$
		&$O'N$	    &~~~~ $\chi_2$ &~~~~  $10944=2^6\cdot3^2\cdot19$\\
		$M_{12}$	&~~~~ $\chi_7$ &~~~~  $54=2\cdot3^3$
		&$Co_3$	&~~~~ $\chi_2$ &~~~~  $23$\\
		$J_1$  	&~~~~ $\chi_4$ &~~~~  $76=2^2\cdot19$
		&$Co_2$	&~~~~ $\chi_2$ &~~~~  $23$\\
		$M_{22}$	&~~~~ $\chi_2$ &~~~~  $21=3\cdot7$
		&$Fi_{22}$  	&~~~~ $\chi_{56}$ &~~~~  $1441792=2^{17}\cdot11$\\
		$J_2$     &~~~~ $\chi_6$ &~~~~  $36=2^2\cdot3^2$
		&$HN$	&~~~~ $\chi_{10}$ &~~~~  $16929=3^4\cdot11\cdot19$\\
		$M_{23}$	&~~~~ $\chi_6$ &~~~~  $22=2\cdot11$
		&$Ly$  	&~~~~ $\chi_7$ &~~~~  $120064=2^8\cdot7\cdot67$\\
		$HS$	    &~~~~ $\chi_2$ &~~~~  $22=2\cdot11$
		&$Th$	    &~~~~ $\chi_2$ &~~~~  $248=2^3\cdot31$\\
		$J_3$	    &~~~~ $\chi_6$ &~~~~  $324=2^3\cdot3^4$    &$Fi_{23}$	    &~~~~ $\chi_4$ &~~~~  $5083=13\cdot17\cdot23$\\
		$M_{24}$	&~~~~ $\chi_2$ &~~~~  $23$
		&$Co_1$	&~~~~ $\chi_3$ &~~~~  $299=13\cdot23$\\
		$M^cL$    &~~~~ $\chi_2$ &~~~~  $22=2\cdot11$
		&$J_4$    &~~~~ $\chi_2$ &~~~~  $1333=31\cdot43$\\
		$He$	    &~~~~ $\chi_9$ &~~~~  $1275=3\cdot5^2\cdot17$   &$Fi_{24}'$	    &~~~~ $\chi_2$ &~~~~  $8671=23\cdot29\cdot13$\\
		$Ru$	    &~~~~ $\chi_5$ &~~~~  $783=3^3\cdot29$
		&$B$	    &~~~~ $\chi_2$ &~~~~  $4371=3\cdot31\cdot47$\\
		$Suz$	    &~~~~ $\chi_2$ &~~~~  $143=11\cdot13$
		&$M$	    &~~~~ $\chi_2$ &~~~~  $196883=59\cdot71\cdot47$\\
		\hline
	\end{tabular}

\end{table}

\end{document}